\documentclass[11pt,letterpaper]{amsart}

\usepackage[letterpaper,margin=1in]{geometry}
\usepackage{amssymb}
\usepackage[numbers,sort&compress]{natbib}
\usepackage{booktabs}
\usepackage{longtable}
\usepackage{array}
\usepackage{hyperref}
\usepackage{xcolor}
\usepackage{listings}
\usepackage{caption}
\hypersetup{
  colorlinks=true,
  linkcolor=blue!60!black,
  citecolor=blue!60!black,
  urlcolor=blue!60!black,
  pdftitle={Split-Cubic Curves from D(n)-Triples: A Bounded Finite-Field Census},
  pdfauthor={Sompong Chuysurichay, Sawian Jaidee, Chatchawan Panraksa, and Teerapol Sukhonwimolmal},
  pdfkeywords={D(n)-triples, Diophantine m-tuples, elliptic curves, Legendre curves, quadratic twists, finite-field point counts, cofactor, embedding degree, finite-field arithmetic diagnostics}
}

\newcommand{\Z}{\mathbb{Z}}
\newcommand{\Q}{\mathbb{Q}}
\newcommand{\F}{\mathbb{F}}
\DeclareMathOperator{\ord}{ord}

\newcommand{\Eabcn}{E_{a,b,c}^{(n)}}

\newtheorem{theorem}{Theorem}[section]

\newtheorem{proposition}[theorem]{Proposition}

\newtheorem{corollary}[theorem]{Corollary}

\theoremstyle{definition}

\newtheorem{definition}[theorem]{Definition}

\newtheorem{remark}[theorem]{Remark}

\newtheorem{question}[theorem]{Question}

\title[Split-Cubic Curves from $D(n)$-Triples]{Split-Cubic Curves from \texorpdfstring{$D(n)$-Triples}{D(n)-Triples}:
       A Bounded Finite-Field Census}

\author[S. Chuysurichay]{Sompong Chuysurichay}
\address{Division of Computational Science, Faculty of Science,
Prince of Songkla University, Songkhla 90110, Thailand}
\email{sompong.c@psu.ac.th}
\author[S. Jaidee]{Sawian Jaidee}
\address{Department of Mathematics, Faculty of Science,
Khon Kaen University, Khon Kaen 40002, Thailand}
\email{jsawia@kku.ac.th}
\thanks{Sawian Jaidee is the corresponding author.}
\author[C. Panraksa]{Chatchawan Panraksa}
\address{Applied Mathematics Program, Mahidol University International
College, Nakhon Pathom 73170, Thailand}
\email{chatchawan.pan@mahidol.ac.th}
\author[T. Sukhonwimolmal]{Teerapol Sukhonwimolmal}
\address{Department of Mathematics, Faculty of Science,
Khon Kaen University, Khon Kaen 40002, Thailand}
\email{teerasu@kku.ac.th}
\date{October 2, 2026}
\begin{document}
\raggedbottom
\begin{abstract}
We record a finite-field census of the split-cubic elliptic curves
$E_{a,b,c}^{(n)}:y^2=(ax+n)(bx+n)(cx+n)$ attached to positive integral
$D(n)$-triples. For $n\in\{-3,3,5,8,12,20\}$ there are $164$ such
triples with $1\le a<b<c\le100$, allowing zero squares. Their
reductions at the ten primes $\operatorname{NextPrime}(10^k)$,
$3\le k\le12$, give $1640$ records, of which $45$ have cofactor four
relative to the largest prime divisor of the group order. We collect
the standard split-cubic and quadratic-twist identities in the
explicit normalization $\lambda=b(c-a)/(c(b-a))$ and
$\delta=cn(b-a)$, and use them to audit the point counts. The archive
records group orders, factorizations, cofactors, embedding degrees
and transfer-field bit lengths. A worked example compares the main
curve with its nonsquare twist. The census is deterministic and
descriptive; it provides neither an asymptotic density estimate nor
cryptographic parameter recommendations.
\end{abstract}
\maketitle

\section{Introduction}
\label{sec:intro}

For nonzero $n$, a $D(n)$-triple is a triple $\{a,b,c\}$ for which
$ab+n$, $ac+n$, and $bc+n$ are squares. We consider the associated
split cubic
\begin{equation}
  \label{eq:Eabcn}
  \Eabcn \;:\; y^2=(ax+n)(bx+n)(cx+n).
\end{equation}
The census in this note contains all positive integral triples with
$1\le a<b<c\le100$ for $n\in\{-3,3,5,8,12,20\}$, under the convention
that zero is a square. There are $164$ such triples. Their reductions
at $p_k=\operatorname{NextPrime}(10^k)$, $3\le k\le12$, give $1640$
triple--prime records, of which $45$ have cofactor $h=4$.
These totals count the specified records; curves are not identified
up to isomorphism.

Elliptic curves attached to Diophantine triples have a substantial
history. Dujella studies the $n=1$ split cubic and its integral
Weierstrass model~\citep{Dujella2001MTuplesEC}; Dujella and
Miki\'c study rational torsion in curves induced by
$D(4)$-triples~\citep{DujellaMikic2014D4}. The rational roots
$-n/a,-n/b,-n/c$ in \eqref{eq:Eabcn} imply full rational
$2$-torsion when $a,b,c$ are nonzero and pairwise distinct.
This property follows from the split form; the $D(n)$-condition
specifies the source of the examples.

The passage from a split cubic to a quadratic twist of a Legendre
curve, and the corresponding quadratic-character point-count
identity, are standard~\citep[Sections~1--2]{AuerTop2002Legendre}.
We record the explicit parameters for \eqref{eq:Eabcn} and the
resulting cofactor test to make the computation transparent.
Related split-cubic twist families and the exclusion of anomalous
reductions also occur in Qiu's work~\citep{Qiu2025Twists}.
For cryptographic examples from Diophantine triples,
Bindu, Somanath, and Das construct a tetradecagonal-number family
and give a finite-field encryption example~\citep{BinduSomanathDas2024}.
Their $D(50)$-triple $\{1,14,31\}$ leads to
$y^2=(x+14)(x+31)(x+434)$ at $p=479$. Our associated integral model
includes the factor $n$ in each constant term, so at $n=50$ it is
the $50$-twist of that cubic. Their single-prime encryption example
and the present multi-parameter census address different questions.

The contribution here is the bounded census and its reproducible
record of point counts, factorizations, Legendre parameters, twist
factors, and cofactor-$4$ cases. The structural identities recalled
below are standard inputs to that computation. The selected height
bound and prime ladder define a finite descriptive dataset; they
do not support an asymptotic density claim or a cryptographic
parameter recommendation. Questions about primes for which
$\#E(\F_p)/t$ is prime already belong to the refined Koblitz
framework~\citep{Zywina2011Koblitz}; the case relevant to $h=4$
has $t=4$.

\section{Background}
\label{sec:background}

\begin{definition}[Positive integral $D(n)$-triple]
  \label{def:dn-triple}
  Let $n\in\Z\setminus\{0\}$. A positive integral $D(n)$-triple is a set
  $\{a,b,c\}\subset \Z_{>0}$ of three pairwise distinct integers such
  that $ab+n$, $ac+n$ and $bc+n$ are all squares in $\Z$. In this
  paper the square $0$ is allowed.
\end{definition}

\begin{remark}[Square-zero convention]
  Some treatments of Diophantine $m$-tuples exclude or separate
  zero-square cases; we include them, while negative integers are not
  squares in $\Z$. This affects the bounded enumeration for $n=-3$
  but not nonsingularity once $nabc(a-b)(a-c)(b-c)\ne0$.
\end{remark}

Diophantine $m$-tuples and their extension problems are classical;
see \citep{Dujella-survey,DujellaFuchs,Dujella-open}. Here the
associated split cubic is used instead for finite-field diagnostics.

\begin{proposition}[Associated split-cubic curve]
  \label{prop:assoc-curve}
  Let $n\in \Z\setminus\{0\}$ and let
  $a,b,c\in\Z_{>0}$ be pairwise distinct. Then the smooth projective
  closure of the affine cubic \eqref{eq:Eabcn} defines an elliptic
  curve $\Eabcn/\Q$. Under the substitution $X = abc\,x$,
  $Y = abc\,y$ it has the integral Weierstrass model
  \begin{equation}
    \label{eq:integral-model}
    Y^{2} \;=\; X^{3} + n(ab+ac+bc)\,X^{2}
            + abc\,n^{2}(a+b+c)\,X + (abc)^{2}n^{3}
  \end{equation}
  The discriminant of the displayed Weierstrass model
  \eqref{eq:integral-model} is
  \begin{equation}
    \label{eq:discriminant}
    \Delta_{\mathrm{mod}}(\Eabcn)
    =16(abc)^{2}n^{6}(a-b)^{2}(a-c)^{2}(b-c)^{2}
  \end{equation}
  The non-trivial rational $2$-torsion points of $\Eabcn$ on this
  model are
  \[
    (-bcn,0),\qquad (-acn,0),\qquad (-abn,0);
  \]
  hence
  $\Eabcn(\Q)[2] \cong \Z/2\Z \oplus \Z/2\Z$.
\end{proposition}

\begin{proof}
  Expanding $(ax+n)(bx+n)(cx+n)$ and substituting $X=abc\,x$,
  $Y=abc\,y$ yields~\eqref{eq:integral-model}. The roots of the cubic
  in $X$ are $-bcn$, $-acn$, $-abn$, with pairwise differences
  $cn(a-b)$, $bn(a-c)$, $an(b-c)$, so the model discriminant is
  \[
    16\{cn(a-b)\}^{2}\{bn(a-c)\}^{2}\{an(b-c)\}^{2},
  \]
  which is~\eqref{eq:discriminant}. Since $nabc\ne 0$ and $a,b,c$ are
  pairwise distinct, this discriminant is nonzero; hence the cubic
  is nonsingular and defines an elliptic curve over $\Q$. The same
  three roots give the stated nontrivial rational $2$-torsion
  points.
\end{proof}

For an odd prime $p$, the displayed model has nonsingular reduction
whenever
\[
  p \nmid abc\,n(a-b)(a-c)(b-c),
\]
and no minimality assertion is made at the excluded primes. The full
rational $2$-torsion used below comes from the split form, not from
the square conditions defining a $D(n)$-triple.

The split-cubic form also gives an explicit Legendre normalization
at every odd prime of good reduction in the sense above. This is
the standard normalization of a split cubic with three rational
roots, together with the usual quadratic-twist ambiguity
\citep[Ch.~III, Sec.~1]{Silverman-AEC}; we recall it only to
identify the parameter and twist factor for the
family~\eqref{eq:Eabcn}. Over a field of odd characteristic, for
\[
  Y^2=(X-r_1)(X-r_2)(X-r_3)
\]
with distinct roots, the substitution
$X=r_1+(r_2-r_1)u$, $Y=(r_2-r_1)Z$ gives
  \[
    Z^2=(r_2-r_1)\,u(u-1)(u-\lambda),
    \qquad
    \lambda=\frac{r_3-r_1}{r_2-r_1}
  \]
Hence the curve is a quadratic twist of the Legendre curve
$L_\lambda:v^2=u(u-1)(u-\lambda)$. For the integral model we take
$r_1=-bcn$, $r_2=-acn$, $r_3=-abn$, so
$r_2-r_1=cn(b-a)$ and $r_3-r_1=bn(c-a)$. Therefore
\[
  \lambda=\frac{r_3-r_1}{r_2-r_1}
  =\frac{b(c-a)}{c(b-a)}
\]
\begin{theorem}[Specialized Legendre--twist normalization]
  \label{thm:legendre-twist}
  Let $n\in \Z\setminus\{0\}$, let $1\le a<b<c$ be integers, and let
  \[
    E=E_{a,b,c}^{(n)}:\quad
    Y^{2}=(X+bcn)(X+acn)(X+abn)
  \]
  be the integral model associated with
  $y^{2}=(ax+n)(bx+n)(cx+n)$. Let $p$ be an odd prime with
  \[
    p\nmid abc\,n(a-b)(a-c)(b-c),
  \]
  (equivalently, the displayed integral model has nonsingular
  reduction at $p$), and let $\overline E_p/\F_p$ be the reduction.
  Let $\chi_p$ denote the quadratic character of $\F_p$, extended
  by $\chi_p(0)=0$. Define
  \[
    \lambda=\frac{b(c-a)}{c(b-a)}\in \F_p\setminus\{0,1\},
    \qquad
    \delta=cn(b-a)\in \F_p^{\times},
  \]
  and let $L_{\lambda}:v^{2}=u(u-1)(u-\lambda)$ be the corresponding
  Legendre curve over $\F_p$. Then the change of variables
  \[
    u=\frac{X+bcn}{cn(b-a)},\qquad
    Z=\frac{Y}{cn(b-a)}
  \]
  identifies $\overline E_p$ over $\F_p$ with
  \[
    C_{\lambda,\delta}:\quad
    Z^{2}=\delta\,u(u-1)(u-\lambda).
  \]
  For an elliptic curve $C/\F_p$, write $a_p(C)=p+1-\#C(\F_p)$. The
  curve $C_{\lambda,\delta}$ is the quadratic twist of $L_\lambda$
  by $\delta$. If $\delta$ is a square, it is $\F_p$-isomorphic to
  $L_\lambda$; if $\delta$ is a nonsquare, it is the twist by the
  nontrivial quadratic character. At exceptional $j$-invariants,
  that twist may still be $\F_p$-isomorphic to $L_\lambda$.
  In either case,
  \[
    a_p(\overline E_p)=\chi_p(\delta)\,a_p(L_{\lambda}),
  \]
  hence
  \[
    \#\overline E_p(\F_p)
    =
    p+1-\chi_p(\delta)\,a_p(L_{\lambda}).
  \]
  In particular, with $a,b,c$ fixed and $p$ as above, the dependence
  on $n$ among values with $p\nmid n$ enters only through
  $\chi_p(n)$.
\end{theorem}

\begin{proof}
  The hypotheses on $p$ make $b-a$, $c-a$, $b-c$, $c$, and $n$
  nonzero in $\F_p$, so $\delta\in\F_p^\times$ and $\lambda\ne 0$.
  If $\lambda=1$, then $b(c-a)=c(b-a)$ in $\F_p$, hence $a(b-c)=0$,
  contradicting the good-reduction hypothesis; thus
  $\lambda\in\F_p\setminus\{0,1\}$.

  Setting $X=-bcn+\delta u$ and $Y=\delta Z$ with $\delta=cn(b-a)$
  gives, over $\F_p$,
  \[
    X+bcn=\delta u,\qquad
    X+acn=\delta(u-1),\qquad
    X+abn=\delta(u-\lambda)
  \]
  The reduced equation becomes
  $\delta^2 Z^{2}=\delta^{3}u(u-1)(u-\lambda)$, and division by
  $\delta^2$ yields
  \[
    Z^{2}=\delta\,u(u-1)(u-\lambda),
  \]
  which is $C_{\lambda,\delta}$. The substitution extends to the
  projective closures by sending the point at infinity of
  $\overline E_p$ to that of $C_{\lambda,\delta}$, giving the
  asserted $\F_p$-isomorphism.

  Counting points on $C_{\lambda,\delta}$ gives
  \[
  \begin{aligned}
    \#\overline E_p(\F_p)
    &=p+1+\sum_{u\in\F_p}
      \chi_p\!\bigl(\delta u(u-1)(u-\lambda)\bigr)\\
    &=p+1+\chi_p(\delta)
      \sum_{u\in\F_p}\chi_p\!\bigl(u(u-1)(u-\lambda)\bigr),
  \end{aligned}
  \]
  since $\delta\in\F_p^\times$. For the Legendre curve,
  \[
    a_p(L_{\lambda})
    =-\sum_{u\in\F_p}\chi_p\!\bigl(u(u-1)(u-\lambda)\bigr),
  \]
  so the trace and point-count formulae follow. Finally,
  $\chi_p(\delta)=\chi_p(c(b-a))\chi_p(n)$, so with $a,b,c,p$ fixed
  the remaining dependence on $n$ among values with $p\nmid n$ is
  through $\chi_p(n)$; this concerns the displayed family only and
  does not assert that $\{a,b,c\}$ remains a $D(n)$-triple as $n$
  varies.
\end{proof}

For what follows, another key consequence of
Proposition~\ref{prop:assoc-curve} is the congruence
\begin{equation}
  \label{eq:four-divides}
  4 \mid \#\overline E_p(\F_{p}) \qquad
  \text{for every odd prime $p$ of good reduction,}
\end{equation}
where $\overline E_p$ denotes the reduction of $\Eabcn$ modulo $p$.
Indeed, when $p$ is odd, the reduction map is injective on
prime-to-$p$ torsion \citep[Prop.~VII.3.1(b)]{Silverman-AEC}, so the
full rational $2$-torsion subgroup
\[
  \Eabcn(\Q)[2] \cong \Z/2\Z \oplus \Z/2\Z
\]
embeds into $\overline E_p(\F_{p})$.

\section{Finite-field arithmetic criteria}
\label{sec:criteria}

For each odd good reduction $\overline E_p/\F_p$ we record the point
count, largest prime divisor $q$, cofactor $h$, embedding degree
$k=\ord_q(p)$, and anomalous status. These are arithmetic diagnostics
only; subgroup structure, CM data, protocol issues, implementation
issues, and standards compliance are outside the scope.

\paragraph{Large prime factor and cofactor.}
Write $\#\overline E_p(\F_p)=hq$ with $q$ the largest prime divisor.
Pollard's rho method~\citep{Pollard-rho} makes the size of $q$
central, while the split form here forces
$4\mid\#\overline E_p(\F_p)$ at every odd good prime. Thus such
reductions are non-anomalous, and for $p\ge 11$ the cofactor relative
to $q$ satisfies $h\ge4$.

\paragraph{Embedding degree.}
If $q=2$, set $k=1$; otherwise let $k=\ord_q(p)$. The MOV and
Frey--R\"uck reductions~\citep{MOV-attack,FrRuck} map into
$\F_{p^k}^{\times}$, so the relevant transfer-field size is measured
by $k\log_2 p$. Since the primes used below satisfy $p\ge1009$, and
if $p\ge 11$ with $p\mid \#\overline E_p(\F_p)$, Hasse's bound gives
\[
  \#\overline E_p(\F_p) \le p+1+2\sqrt p < 2p,
\]
the non-anomalous reductions considered here have $q\ne p$. The Magma
scripts compute this order exactly.

\paragraph{Anomalous reduction.}
A curve is \emph{anomalous} at $p$ if $\#\overline E_p(\F_p)=p$;
Smart and Semaev solve the discrete logarithm on such curves in
polynomial time~\citep{Smart-anomalous,SemaevAnomalous}. Here odd good
reductions are automatically non-anomalous because
$4\mid\#\overline E_p(\F_p)$.

The next proposition is stated for the $D(n)$-triple family, but its
proof uses only the split-cubic model with three distinct roots and
good odd reduction.

\begin{proposition}[Automatic structural properties at odd good primes]
  \label{prop:forced-cofactor}
  Let $n\in\Z\setminus\{0\}$, let $\{a,b,c\}$ be a positive
  integral $D(n)$-triple, and let $p$ be an odd
  prime of good reduction for $\Eabcn$, with reduction
  $\overline E_p/\F_p$. Then
  \begin{enumerate}
    \item $4 \mid \#\overline E_p(\F_{p})$;
    \item $\overline E_p$ is non-anomalous.
  \end{enumerate}
  If moreover $p \ge 11$, then $\#\overline E_p(\F_{p}) > 4$ by
  Hasse's bound, so with $q$ the largest prime divisor of
  $\#\overline E_p(\F_{p})$ and $h=\#\overline E_p(\F_{p})/q$ one has
  $h \ge 4$. The primes $p=3,5,7$ are not excluded from~(i)
  or~(ii); they are excluded only from this cofactor lower bound.
\end{proposition}

\begin{proof}
  Since $p$ is odd, Equation~\eqref{eq:four-divides} gives~(i). If
  $\#\overline E_p(\F_{p}) = p$, then $4 \mid p$, impossible for odd
  $p$, giving~(ii). For $p \ge 11$, Hasse's bound gives
  $\#\overline E_p(\F_{p}) \ge p + 1 - 2\sqrt{p} > 4$. Set
  $N=\#\overline E_p(\F_{p})$ and let $q$ be the largest prime
  divisor of $N$. If $q=2$, then $N$ is a power of $2$, and $N>4$
  forces $N\ge 8$, hence $h=N/2\ge 4$. If $q$ is odd, then
  $4\mid N/q=h$, so again $h\ge 4$.
\end{proof}

Thus, at the odd primes used below, $4$ is the smallest possible
cofactor in this family.

\begin{corollary}[Cofactor-four criterion from the Legendre trace]
  \label{cor:h4-criterion}
  With the notation of Theorem~\ref{thm:legendre-twist} and $p\ge 11$, the
  reduction $\overline E_p/\F_p$ has cofactor
  $h(\overline E_p/\F_p)=4$ (relative to the largest prime divisor of
  $\#\overline E_p(\F_p)$) if and only if
  \[
    \frac{p+1-\chi_p(\delta)\,a_p(L_{\lambda})}{4}
  \]
  is prime ($2$ included). If $\overline E'_p$ denotes the
  quadratic twist of $\overline E_p$ by any nonsquare in
  $\F_p^\times$, then
  $h(\overline E'_p/\F_p)=4$ if and only if
  \[
    \frac{p+1+\chi_p(\delta)\,a_p(L_{\lambda})}{4}
  \]
  is prime ($2$ included).
\end{corollary}

\begin{proof}
  By Theorem~\ref{thm:legendre-twist},
  $\#\overline E_p(\F_p)=p+1-\chi_p(\delta)\,a_p(L_{\lambda})$. The
  displayed quotient is an integer because the split model has
  three distinct rational points of order $2$. Since $p\ge11$,
  Hasse's bound gives $\#\overline E_p(\F_p)>4$. If the cofactor is
  $4$, then the quotient is its largest prime divisor. Conversely,
  if $\#\overline E_p(\F_p)=4r$ with $r$ prime, then $r$ is the
  largest prime divisor, including when $r=2$, so the cofactor is $4$.

  The quadratic twist by a nonsquare has the opposite Frobenius trace
  and order $p+1+\chi_p(\delta)\,a_p(L_{\lambda})$. Explicitly, for
  any nonsquare $d\in\F_p^\times$, the twist by $d$ of
  $C_{\lambda,\delta}:Z^2=\delta\,u(u-1)(u-\lambda)$ is represented
  by $Z^2=d\delta\,u(u-1)(u-\lambda)$. The points $(0,0)$, $(1,0)$,
  $(\lambda,0)$, $O$ are still $\F_p$-rational, so the same
  $2$-torsion argument applies; the twist order is divisible by $4$
  and greater than $4$ for $p\ge 11$, giving the second condition.
\end{proof}

Proposition~\ref{prop:forced-cofactor} and Corollary~\ref{cor:h4-criterion} set the agenda for the
rest of the paper: at our ten decimal-size odd test primes, we
measure how often the minimal cofactor $h=4$ is attained, with each
such occurrence detected by the primality of the displayed quotient.

\section{Enumeration and algorithm}
\label{sec:algorithm}

The survey enumerates positive integral $D(n)$-triples with entries at
most $H$ by an $O(H^3)$ triple loop and integer square tests in $\Z$,
with $0$ accepted by Definition~\ref{def:dn-triple}.

\begin{center}
\begin{minipage}{0.92\linewidth}
\begin{lstlisting}[language=Pascal]
for a := 1 to H do
  for b := a+1 to H do
    if IsIntegerSquare(a*b + n) then
      for c := b+1 to H do
        if IsIntegerSquare(a*c + n) and IsIntegerSquare(b*c + n) then
          record (a, b, c)
\end{lstlisting}
\end{minipage}
\end{center}
Table~\ref{tab:enum-counts} summarises the triple counts up to $H = 100$
for the six values of $n$ used below; the script records the count,
first triple, and last triple.

\begin{table}[t]
\centering
\caption{\label{tab:enum-counts}
Number of integral $D(n)$-triples $\{a,b,c\}$ with
$1 \le a<b<c \le H = 100$, under the square-zero convention of
Definition~\ref{def:dn-triple}.}
\begin{tabular}{rrll}
\toprule
$n$ & count & first triple & last triple \\
\midrule
$-3$ & $39$ & $\{1,3,4\}$   & $\{19,28,93\}$ \\
$3$  & $21$ & $\{1,6,13\}$  & $\{13,37,94\}$ \\
$5$  & $23$ & $\{1,4,11\}$  & $\{11,29,76\}$ \\
$8$  & $26$ & $\{1,8,17\}$  & $\{17,28,89\}$ \\
$12$ & $30$ & $\{1,4,13\}$  & $\{13,33,88\}$ \\
$20$ & $25$ & $\{1,5,16\}$  & $\{19,20,79\}$ \\
\midrule
\textbf{total} & $164$ & & \\
\bottomrule
\end{tabular}
\end{table}

For the representative-curve computations of
Table~\ref{tab:crypto-signatures}, we use one explicit triple per $n$.
For $n=3,5,8,12,20$ we take the lexicographically first triple in
Table~\ref{tab:enum-counts}. For $n=-3$ we instead use $(1,4,7)$, the
third enumerated $D(-3)$-triple and the first for which $ab+n$,
$ac+n$, $bc+n$ are all positive squares; this avoids making the
representative depend on a zero-square case. These six examples are
representatives only. The exhaustive part is the separate
$164\times10$ scan summarized in Table~\ref{tab:h4-hits}; full $h=4$ rows
are deposited in the reproducibility archive.

For a fixed triple $(a,b,c)$ we build the integral
model~\eqref{eq:integral-model} in Magma V2.29-6~\citep{Magma}, reduce
it at the ten decimal-size test primes
\[
  p_k=\operatorname{NextPrime}(10^k),\qquad 3\le k\le 12,
\]
via \texttt{ChangeRing}, and read the group order from
\texttt{\#Ep}. Exact point counting over finite fields is standard,
with Schoof's algorithm as the classical polynomial-time
reference~\citep{Schoof}, and Magma supplies the implementation used
here. By~Theorem~\ref{thm:legendre-twist}, the same order can equivalently
be written as
\[
  p+1-\chi_p\!\bigl(cn(b-a)\bigr)\,
  a_p\!\bigl(L_{b(c-a)/(c(b-a))}\bigr),
\]
so the tables below may be read as point-count data for these
explicit Legendre twists. The prime ladder is reproducible rather
than intrinsic: it samples one odd prime just above each power of
$10$ from $10^3$ through $10^{12}$, with no tuning to any particular
triple.

In the bounded scan each of $a,b,c,n,a-b,a-c,b-c$ has nonzero
absolute value strictly smaller than $1009$, so no selected prime
divides any of these factors, and hence none divides
$abc\,n(a-b)(a-c)(b-c)$. By~\eqref{eq:discriminant}, all reductions
in Tables~\ref{tab:crypto-signatures} and~\ref{tab:h4-hits} are at odd primes of
good reduction.

The representative computations in Table~\ref{tab:crypto-signatures} are
recorded in the archived computation transcript: for each reduction we factor
$\#\overline E_p(\F_p)$, set $q$ to its largest prime divisor,
define $h=\#\overline E_p(\F_p)/q$, and compute
$k=\ord_q(p)$ using Magma's multiplicative-order routine. The
exhaustive $164\times10$ scan summarized in Table~\ref{tab:h4-hits} is
archived separately with a full row-level data export, in which
the columns \texttt{lambda} and \texttt{delta} give the least
nonnegative integer representatives of the corresponding classes in
$\F_p$, and \texttt{factorization\_N} records the full prime
factorization of $N=\#\overline E_p(\F_p)$.

The scripts expose $H$, the list of $n$, and the prime list for larger
experiments; the present tables are the finite screen described above.

\section{Computational results}
\label{sec:results}

Table~\ref{tab:crypto-signatures} reports finite-field arithmetic data for
each representative reduction: point count, cofactor $h$, largest
prime factor $q$ with bit length, and embedding degree
$k=\ord_q(p)$. For MOV/Frey--R\"uck interpretation we also use
\[
  b_{\mathrm{MOV}}=\lfloor k\log_2 p\rfloor+1,
\]
the bit length of the transfer-field size $p^k$. The reported
$b_{\mathrm{MOV}}$ values were computed from this formula in the
export script with 200-bit real precision; they are used only as
transfer-field-size summaries and not in any proof. The rows are
descriptive data from the deterministic grid of
Section~\ref{sec:algorithm}; they are not sampled from a probability model
and should not be read as estimating an asymptotic density for
$h=4$.

All rows occur at odd primes of good reduction by the criterion
verified in Section~\ref{sec:algorithm}. The main-text tables give the
representative reductions and a summary of all cofactor-four
hits; the archived data contain the full row-level audit
data, including the factorization of each order
$N=\#\overline E_p(\F_p)$. Table~\ref{tab:audit-summary} records row
counts and extremal values.
By~Theorem~\ref{thm:legendre-twist}, each point count is the order of an
explicit quadratic twist of the Legendre curve with parameter
$\lambda=b(c-a)/(c(b-a))$; the representative computation records
$\lambda$, $\delta$, $\chi_p(\delta)$, $a_p(L_\lambda)$ and asserts
the Legendre trace formula for every row in
Table~\ref{tab:crypto-signatures}. Twist data are not included in this
batch table; they are treated separately in the worked example of
Section~\ref{sec:worked}. Every representative reduction is non-anomalous
across all ten primes, as predicted by~Proposition~\ref{prop:forced-cofactor}.

\begin{table}[t]
\centering\small
\setlength{\tabcolsep}{4pt}
\caption{\label{tab:audit-summary}
Audit summary for the finite computational screen.}
\begin{tabular}{lrr}
\toprule
Data set & rows & \(h=4\) rows \\
\midrule
All scanned reductions & \(1640\) & \(45\) \\
Cofactor-four subtable & \(45\) & \(45\) \\
\bottomrule
\end{tabular}

\medskip
\begin{tabular}{lll}
\toprule
Quantity & value & attaining row \((n,a,b,c,p)\) \\
\midrule
Maximum \(q\) & \(249{,}999{,}945{,}919\)
& \((20,4,11,31,1000000000039)\) \\
Maximum \(b_{\mathrm{MOV}}\) & \(9{,}965{,}782{,}128{,}798\)
& \((20,4,11,31,1000000000039)\) \\
\bottomrule
\end{tabular}

\end{table}

\begingroup
\small
\setlength{\tabcolsep}{2.5pt}
\begin{longtable}{rlrrrrrr}
\caption{\label{tab:crypto-signatures}
Finite-field arithmetic data for representative reductions
$\overline E_p$ at ten decimal-size test primes. The ``bits'' column
gives the binary bit-length $\lfloor \log_{2} q \rfloor+1$; the $k$
column gives the exact embedding degree $\ord_q(p)$. The associated
transfer-field bit length is
$b_{\mathrm{MOV}}=\lfloor k\log_2 p\rfloor+1$, discussed below.
Extracted from the archived computation transcript.}\\
\toprule
$n$ & triple & $p$ & $\#\overline E_p(\F_{p})$ & $h$ & $q$ & bits & $k$ \\
\midrule
\endfirsthead
\caption[]{Finite-field arithmetic data for representative
reductions $\overline E_p$ at ten decimal-size test primes
(continued).}\\
\toprule
$n$ & triple & $p$ & $\#\overline E_p(\F_{p})$ & $h$ & $q$ & bits & $k$ \\
\midrule
\endhead
\bottomrule
\endlastfoot
$-3$ & $(1,4,7)$ & $1009$    & $1008$    & $144$  & $7$      & $3$  & $1$   \\
     &           & $10007$   & $9880$    & $520$  & $19$     & $5$  & $18$  \\
     &           & $100003$  & $99856$   & $1264$ & $79$     & $7$  & $78$  \\
     &           & $1000003$ & $1000680$ & $3720$ & $269$    & $9$  & $268$ \\
     &           & $10000019$ & $9997248$ & $192$ & $52069$ & $16$ & $52068$ \\
     &           & $100000007$ & $99987520$ & $13120$ & $7621$ & $13$ & $7620$ \\
     &           & $1000000007$ & $1000035032$ & $12328$ & $81119$ & $17$ & $81118$ \\
     &           & $10000000019$ & $10000116216$ & $72$ & $138890503$ & $28$ & $138890502$ \\
     &           & $100000000003$ & $100000469072$ & $293776$ & $340397$ & $19$ & $24314$ \\
     &           & $1000000000039$ & $1000000748400$ & $1200$ & $833333957$ & $30$ & $833333956$ \\
\midrule
$3$  & $(1,6,13)$ & $1009$    & $1048$    & $8$    & $131$    & $8$  & $26$  \\
     &            & $10007$   & $9900$    & $900$  & $11$     & $4$  & $10$  \\
     &            & $100003$  & $100112$  & $16$   & $6257$   & $13$ & $6256$ \\
     &            & $1000003$ & $1000768$ & $1216$ & $823$    & $10$ & $822$ \\
     &            & $10000019$ & $9998480$ & $80$ & $124981$ & $17$ & $41660$ \\
     &            & $100000007$ & $99994612$ & $12988$ & $7699$ & $13$ & $7698$ \\
     &            & $1000000007$ & $999992320$ & $5120$ & $195311$ & $18$ & $39062$ \\
     &            & $10000000019$ & $10000149816$ & $123384$ & $81049$ & $17$ & $27016$ \\
     &            & $100000000003$ & $99999835152$ & $685104$ & $145963$ & $18$ & $145962$ \\
     &            & $1000000000039$ & $999998709456$ & $912$ & $1096489813$ & $31$ & $365496604$ \\
\midrule
$5$  & $(1,4,11)$ & $1009$    & $1008$    & $144$  & $7$      & $3$  & $1$   \\
     &            & $10007$   & $10104$   & $24$   & $421$    & $9$  & $210$ \\
     &            & $100003$  & $99732$   & $12$   & $8311$   & $14$ & $4155$ \\
     &            & $1000003$ & $1000056$ & $24$   & $41669$  & $16$ & $20834$ \\
     &            & $10000019$ & $9997896$ & $24$ & $416579$ & $19$ & $416578$ \\
     &            & $100000007$ & $99983240$ & $280$ & $357083$ & $19$ & $357082$ \\
     &            & $1000000007$ & $1000054484$ & $33748$ & $29633$ & $15$ & $29632$ \\
     &            & $10000000019$ & $10000074832$ & $56848$ & $175909$ & $18$ & $175908$ \\
     &            & $100000000003$ & $100000028556$ & $36$ & $2777778571$ & $32$ & $1388889285$ \\
     &            & $1000000000039$ & $999998738600$ & $200$ & $4999993693$ & $33$ & $2499996846$ \\
\midrule
$8$  & $(1,8,17)$ & $1009$    & $1016$    & $8$    & $127$    & $7$  & $63$ \\
     &            & $10007$   & $10120$   & $440$  & $23$     & $5$  & $11$  \\
     &            & $100003$  & $100400$  & $400$  & $251$    & $8$  & $125$ \\
     &            & $1000003$ & $998908$  & $4$    & $249727$ & $18$ & $41621$ \\
     &            & $10000019$ & $9999852$ & $13908$ & $719$ & $10$ & $359$ \\
     &            & $100000007$ & $100006880$ & $203680$ & $491$ & $9$ & $1$ \\
     &            & $1000000007$ & $999979976$ & $8$ & $124997497$ & $27$ & $62498748$ \\
     &            & $10000000019$ & $9999809320$ & $40$ & $249995233$ & $28$ & $31249404$ \\
     &            & $100000000003$ & $99999894276$ & $4448196$ & $22481$ & $15$ & $5620$ \\
     &            & $1000000000039$ & $999999512376$ & $19730472$ & $50683$ & $16$ & $25341$ \\
\midrule
$12$ & $(1,4,13)$ & $1009$    & $1016$    & $8$    & $127$    & $7$  & $63$ \\
     &            & $10007$   & $9840$    & $240$  & $41$     & $6$  & $8$   \\
     &            & $100003$  & $99804$   & $12$   & $8317$   & $14$ & $4158$ \\
     &            & $1000003$ & $1000504$ & $8$    & $125063$ & $17$ & $62531$ \\
     &            & $10000019$ & $9996776$ & $6616$ & $1511$ & $11$ & $151$ \\
     &            & $100000007$ & $99998800$ & $4400$ & $22727$ & $15$ & $11363$ \\
     &            & $1000000007$ & $999937272$ & $109992$ & $9091$ & $14$ & $505$ \\
     &            & $10000000019$ & $9999902088$ & $87912$ & $113749$ & $17$ & $37916$ \\
     &            & $100000000003$ & $100000243244$ & $4532$ & $22065367$ & $25$ & $22065366$ \\
     &            & $1000000000039$ & $999998879904$ & $689179104$ & $1451$ & $11$ & $725$ \\
\midrule
$20$ & $(1,5,16)$ & $1009$    & $984$     & $24$   & $41$     & $6$  & $10$  \\
     &            & $10007$   & $9992$    & $8$    & $1249$   & $11$ & $39$  \\
     &            & $100003$  & $100300$  & $1700$ & $59$     & $6$  & $29$  \\
     &            & $1000003$ & $999784$  & $376$  & $2659$   & $12$ & $1329$ \\
     &            & $10000019$ & $10002640$ & $7760$ & $1289$ & $11$ & $184$ \\
     &            & $100000007$ & $99994456$ & $8$ & $12499307$ & $24$ & $6249653$ \\
     &            & $1000000007$ & $1000041192$ & $72$ & $13889461$ & $24$ & $13889460$ \\
     &            & $10000000019$ & $10000137456$ & $2736$ & $3655021$ & $22$ & $3655020$ \\
     &            & $100000000003$ & $100000146188$ & $236$ & $423729433$ & $29$ & $47081048$ \\
     &            & $1000000000039$ & $999998772068$ & $1035724$ & $965507$ & $20$ & $482753$ \\
\end{longtable}
\endgroup

\begin{table}[htbp]
\centering
\small
\setlength{\tabcolsep}{3pt}
\caption{\label{tab:mov-summary}
Summary of transfer-field bit lengths
\(b_{\mathrm{MOV}}=\lfloor k\log_2 p\rfloor+1\) for the two displayed
data sets.}
\begin{tabular}{lrrl}
\toprule
Data set & rows & range of \(b_{\mathrm{MOV}}\)
& max row \((n,a,b,c,p)\) \\
\midrule
Representatives
& $60$ & $10$--$99{,}657{,}717{,}119$
& $(5,1,4,11,1000000000039)$ \\
\(h=4\) rows
& $45$ & $250$--$9{,}965{,}782{,}128{,}798$
& $(20,4,11,31,1000000000039)$ \\
\bottomrule
\end{tabular}
\end{table}

Two observations deserve emphasis. First, none of the six
representative triples attains $h = 4$ at $p=1009$ or $p=10007$, but
$(1,8,17)$ with $n=8$ does attain $h=4$ at $p = 1000003$, producing
an $18$-bit prime $q$; this is the worked example of
Section~\ref{sec:worked}. Second, the embedding degrees vary widely, from
$k=1$ and $k=8$ up to $k=2{,}499{,}996{,}846$, with corresponding
transfer-field bit lengths summarised in Table~\ref{tab:mov-summary}.
These are exact finite-field invariants of the sampled reductions,
not deployment-level security estimates: the prime subgroup sizes
in the table are far below modern ECC sizes.

A full scan of the $164$ triples in Table~\ref{tab:enum-counts} over the
ten test primes gives forty-five $h=4$ cases out of $1640$ rows. By
Corollary~\ref{cor:h4-criterion}, these are exactly the rows for which
\[
  \frac{p+1-\chi_p(\delta)\,a_p(L_{\lambda})}{4}
\]
is prime. Table~\ref{tab:h4-hits} summarizes the hits; the full row list,
including \(N\), factorization, \(\lambda\), \(\delta\),
\(\chi_p(\delta)\), \(a_p(L_\lambda)\), \(q\), \(h\), \(k\), and
\(b_{\mathrm{MOV}}\), is deposited in the reproducibility archive.

\begin{table}[t]
\centering\small
\caption{\label{tab:h4-hits}
Summary of all cofactor-four rows in the bounded scan, with
\(p_k=\operatorname{NextPrime}(10^k)\).}
\begin{tabular}{lrrrrrr}
\toprule
\(n\) & \(-3\) & \(3\) & \(5\) & \(8\) & \(12\) & \(20\) \\
\midrule
\(\#\{h=4\}\) & \(10\) & \(2\) & \(6\) & \(9\) & \(9\) & \(9\) \\
\bottomrule
\end{tabular}

\medskip
\begin{tabular}{lrrrrr}
\toprule
\(p_k\) & \(p_3\) & \(p_4\) & \(p_5\) & \(p_6\) & \(p_7\) \\
\midrule
\(\#\{h=4\}\) & \(18\) & \(8\) & \(7\) & \(3\) & \(4\) \\
\bottomrule
\end{tabular}

\medskip
\begin{tabular}{lrrrrr}
\toprule
\(p_k\) & \(p_8\) & \(p_9\) & \(p_{10}\) & \(p_{11}\) & \(p_{12}\) \\
\midrule
\(\#\{h=4\}\) & \(2\) & \(0\) & \(1\) & \(0\) & \(2\) \\
\bottomrule
\end{tabular}

\medskip
\begin{tabular}{lll}
\toprule
Extremal row & value & \((n,a,b,c,p)\) \\
\midrule
Largest \(q\) & \(249{,}999{,}945{,}919\)
& \((20,4,11,31,1000000000039)\) \\
Next largest \(q\) & \(249{,}999{,}898{,}297\)
& \((5,4,29,55,1000000000039)\) \\
\bottomrule
\end{tabular}
\end{table}

\section{Worked example:
  \texorpdfstring{$\{1,8,17\}$ with $n=8$}{\{1,8,17\} with n=8}}
\label{sec:worked}

We examine the only representative in
Table~\ref{tab:crypto-signatures} whose associated curve attains $h=4$ at
one of our ten test primes: the $D(8)$-triple $\{1,8,17\}$. The
example shows how the Legendre--twist formula certifies a
cofactor-four row. The arithmetic verification is immediate:
$1\cdot 8+8 = 16 = 4^{2}$,
$1\cdot 17+8 = 25 = 5^{2}$, $8\cdot 17+8 = 144 = 12^{2}$. Expanding
$(x+8)(8x+8)(17x+8)$ gives
\[
  f(x) = 136\,x^{3} + 1288\,x^{2} + 1664\,x + 512,
\]
and the substitution $X=abc\,x=136 x$, $Y=abc\,y=136 y$ produces
the integral Weierstrass model of Proposition~\ref{prop:assoc-curve},
\begin{equation}
  \label{eq:E1817}
  E_{1,8,17}^{(8)}\;:\; Y^{2} = X^{3} + 1288\,X^{2}
                       + 226{,}304\,X + 9{,}469{,}952
\end{equation}

The model discriminant of~\eqref{eq:E1817} is
\[
  \Delta_{\mathrm{mod}}=78\,824\,057\,314\,738\,176
  =2^{36}\cdot 3^{4}\cdot 7^{2}\cdot 17^{2},
\]
None of the ten test primes divides this discriminant, so the
displayed model has good reduction at all ten primes. By
Proposition~\ref{prop:assoc-curve}, the three nontrivial rational $2$-torsion
points are
\[
  \begin{aligned}
    P_{a}&=(-bcn,0)=(-1088,0),\\
    P_{b}&=(-acn,0)=(-136,0),\\
    P_{c}&=(-abn,0)=(-64,0)
  \end{aligned}
\]
so $E_{1,8,17}^{(8)}(\Q)[2]\cong \Z/2\Z\oplus \Z/2\Z$.

For this example, Theorem~\ref{thm:legendre-twist} gives
\[
  \lambda=\frac{8(17-1)}{17(8-1)}=\frac{128}{119},
  \qquad
  \delta=17\cdot 8\cdot (8-1)=952
\]
Hence, with $\lambda,\delta$ reduced modulo $p$ and $\overline E_p$
the reduction of~\eqref{eq:E1817} at any test prime,
\[
  \#\overline E_p(\F_p)=p+1-\chi_p(952)\,a_p(L_{128/119}),
\]
and the quadratic twist by a nonsquare has order
$p+1+\chi_p(952)\,a_p(L_{128/119})$. Equivalently, the Frobenius
trace $t=p+1-\#\overline E_p(\F_p)$ equals
$\chi_p(952)\,a_p(L_{128/119})$.

Tables~\ref{tab:worked-reductions} and~\ref{tab:worked-twists} report the main
reductions and their nonsquare-twist orders, respectively.
Table~\ref{tab:worked-reductions} lists $\chi_p(952)$, $a_p(L_{128/119})$,
the trace $t=\chi_p(952)\,a_p(L_{128/119})$, and
$\#\overline E_p(\F_p)$. In the Magma script, $\overline E'_p$ is
constructed as $E^{(d_p)}$ for an explicit nonsquare
$d_p\in\F_p^\times$, and the checks
\[
  t=\chi_p(952)\,a_p(L_{128/119}),
  \qquad
  \#\overline E_p(\F_p)+\#\overline E'_p(\F_p)=2p+2
\]
are asserted. As an independent small-prime check, the script also
evaluates the character sum
\[
  p+1+\sum_{u\in\F_p}
  \chi_p\!\bigl(952\,u(u-1)(u-128/119)\bigr)
\]
for $p=1009,10007,100003$ and compares it with
$\#\overline E_p(\F_p)$. We do not identify $\overline E'_p$ with
$E^{(-1)}$; for example, $-1$ is a square in $\F_{1009}$.

\begin{table}[t]
\centering\small
\setlength{\tabcolsep}{3pt}
\caption{\label{tab:worked-reductions}
Main-reduction data for $E_{1,8,17}^{(8)}$ at the ten decimal-size
test primes. Here $L=L_{128/119}$ and
$t=\chi_p(952)a_p(L)$. Boldface flags the cofactor-four row.}
\begin{tabular}{rrrrrrrr}
\toprule
$p$ & $\chi_p(952)$ & $a_p(L)$ & $t$
& $\#\overline E_p(\F_p)$ & $h$ & $q$ & $k$ \\
\midrule
$1009$    & $-1$ & $6$     & $-6$   & $1016$    & $8$          & $127$       & $63$ \\
$10007$   & $1$  & $-112$  & $-112$ & $10120$   & $440$        & $23$        & $11$ \\
$100003$  & $1$  & $-396$  & $-396$ & $100400$  & $400$        & $251$       & $125$ \\
$1000003$ & $-1$ & $-1096$ & $1096$ & $998908$  & $\mathbf{4}$ & $249727$    & $41621$ \\
$10000019$ & $-1$ & $-168$ & $168$ & $9999852$ & $13908$ & $719$ & $359$ \\
$100000007$ & $1$ & $-6872$ & $-6872$ & $100006880$ & $203680$ & $491$ & $1$ \\
$1000000007$ & $-1$ & $-20032$ & $20032$ & $999979976$ & $8$ & $124997497$ & $62498748$ \\
$10000000019$ & $1$ & $190700$ & $190700$ & $9999809320$ & $40$ & $249995233$ & $31249404$ \\
$100000000003$ & $-1$ & $-105728$ & $105728$ & $99999894276$ & $4448196$ & $22481$ & $5620$ \\
$1000000000039$ & $1$ & $487664$ & $487664$ & $999999512376$ & $19730472$ & $50683$ & $25341$ \\
\bottomrule
\end{tabular}
\end{table}

\begin{table}[t]
\centering\small
\setlength{\tabcolsep}{5pt}
\caption{\label{tab:worked-twists}
Orders of the quadratic twists by nonsquares for the same worked example.
Boldface flags the twist cofactor-four row.}
\begin{tabular}{rrr}
\toprule
$p$ & $\#\overline E'_p(\F_p)$ & $h'$ \\
\midrule
$1009$    & $1004$    & $\mathbf{4}$ \\
$10007$   & $9896$    & $8$ \\
$100003$  & $99608$   & $8$ \\
$1000003$ & $1001100$ & $14100$ \\
$10000019$ & $10000188$ & $396$ \\
$100000007$ & $99993136$ & $16$ \\
$1000000007$ & $1000020040$ & $1720$ \\
$10000000019$ & $10000190720$ & $16640$ \\
$100000000003$ & $100000105732$ & $482908$ \\
$1000000000039$ & $1000000487704$ & $8103272$ \\
\bottomrule
\end{tabular}
\end{table}

The ten trace rows make the relation
$t=\chi_p(952)\,a_p(L_{128/119})$ explicit; for instance, at
$p=1000003$ one has $\chi_p(952)=-1$ and $a_p(L_{128/119})=-1096$,
giving $t=1096$.

The example illustrates how the structural criterion detects an
$h=4$ row. At $p=1000003$ the trace is $t=1096$, so
\[
  \#\overline E_p(\F_p)=p+1-t=998{,}908=4\cdot 249{,}727
\]
Equivalently, $(p+1-t)/4=249{,}727$ is prime, matching
Corollary~\ref{cor:h4-criterion}. The cofactor is $h=4$, the largest prime
factor $q=249{,}727$ is an $18$-bit prime, the embedding degree is
$k=41621$, and the reduction is non-anomalous
by~Proposition~\ref{prop:forced-cofactor}. At the smaller prime $p=1009$ the
same formula explains the twist hit: $t=-6$, so the nonsquare
quadratic twist has
\[
  \#\overline E'_p(\F_p)=p+1+t=1004=4\cdot 251,
\]
hence $h'=4$ with twist prime $q'=251$.

These checks are deliberately narrow. The $18$-bit value of $q$ at
$p=1000003$ is far below any deployable security level: modern
prime-order or large-prime-subgroup ECC deployments use subgroup
sizes on the order of hundreds of bits. The example is also not
twist-secure at $p=1000003$, since the quadratic twist by a nonsquare
has order
\[
  \#\overline E'_p(\F_{p})=1{,}001{,}100=14{,}100\cdot 71,
\]
so its largest prime factor is only $71$. The row is positive only
for the specific arithmetic filters verified here: minimal cofactor
on the main reduction, prime largest factor, exact embedding degree,
and absence of anomalous reduction. It is not a cryptographic
suitability claim, and it illustrates why cofactor, embedding-degree,
and anomalous-reduction checks are necessary but far from
sufficient.

The ancillary scripts allow larger primes and height bounds to be
tested in the same pipeline. The frequency of cofactor-four reductions
as these bounds grow is a separate question, formulated in
Question~\ref{q:asymptotic-h4}. The existence of substantially larger prime
subgroups would require its own analysis.

\section{Discussion and open questions}
\label{sec:discussion}

The screen above is finite: $H=100$, ten deterministic test primes,
and subgroup sizes far below deployment-scale ECC. Within that
screen the structural lower bound $h\ge 4$ is sharp: forty-five of
the $1640$ deterministic (triple, prime) pairs have $h=4$. To ask for
an asymptotic analogue one must first specify the sampling model. For
fixed $n\in\Z\setminus\{0\}$, let $\mathcal S_n(H,X)$ be the set of quadruples $(a,b,c,p)$
with $1\le a<b<c\le H$, $p\le X$ an odd prime,
$p\nmid\Delta_{\mathrm{mod}}(\Eabcn)$, and $\{a,b,c\}$ an integral
$D(n)$-triple.
When $\mathcal S_n(H,X)\ne\varnothing$, define
\[
R_n(H,X)=
\frac{
\#\{(a,b,c,p)\in\mathcal S_n(H,X):h(\overline E_p/\F_p)=4\}
}{
\#\mathcal S_n(H,X)
},
\]
where $\overline E_p$ denotes the reduction of $E_{a,b,c}^{(n)}$
modulo $p$ and the cofactor is taken relative to the largest prime
divisor of its group order.

For fixed $H$ containing at least one such triple, the limit in the
prime bound is already determined:
\begin{equation}
  \label{eq:fixed-H-limit}
  \lim_{X\to\infty}R_n(H,X)=0.
\end{equation}
This is a standard consequence of finite congruence sieves for a fixed
elliptic curve, rather than an asymptotic suggested by our data.
For each fixed curve $E/\Q$, the good primes for which
$\#\overline E_p(\F_p)/4$ is an integer and prime have density zero
relative to the primes.\footnote{For a fixed sieve bound $y$, all but
finitely many such primes pass the tests excluding every prime divisor
$\ell\le y$ from $\#\overline E_p(\F_p)/4$; the exceptions are finite
by Hasse's bound. Chebotarev gives the density of this finite set of
congruence tests. For a non-CM curve, these densities tend to zero by
\citep[Sec.~2.1 and Prop.~2.4]{Zywina2011Koblitz}. For a CM curve,
apply \citep[Prop.~2.8]{Zywina2011Koblitz} over its CM field to obtain
the same conclusion for split rational primes. For inert primes,
$\#\overline E_p(\F_p)=p+1$
\citep[Sec.~2.3.1]{Zywina2011Koblitz}; the primes for which
$(p+1)/4$ is prime must, apart from finitely many exceptions, avoid
$p\equiv-1\pmod\ell$ for every odd $\ell\le y$. The density of the inert target set is bounded above by the
density of all rational primes passing these tests, namely
$\prod_{\substack{3\le\ell\le y\\\ell\ \mathrm{prime}}}(\ell-2)/(\ell-1)$, which tends to zero.
In each case take $X\to\infty$ at fixed $y$ first, then
$y\to\infty$. No Koblitz asymptotic or uniform error term is used.}
If $T_n(H)>0$ is the number of triples with $1\le a<b<c\le H$,
the denominator defining $R_n(H,X)$ is
$T_n(H)\pi(X)+O_{n,H}(1)$, while its numerator is $o_{n,H}(\pi(X))$.
Here we use the cofactor-four criterion for $p\ge11$; the finitely
many smaller or model-bad primes do not affect the limit.
Consequently, for any fixed $n$ admitting at least one positive
integral $D(n)$-triple,
\[
  \lim_{H\to\infty}\left(\lim_{X\to\infty}R_n(H,X)\right)=0.
\]
This argument provides no estimate uniform in $H$. Reversing the
limits or allowing $H$ to grow with $X$ therefore requires a separate
analysis.

\begin{question}
  \label{q:asymptotic-h4}
  Fix $n\in\Z\setminus\{0\}$ admitting a positive integral
  $D(n)$-triple. For every sufficiently large fixed $X$, does
  \[
    \lim_{H\to\infty}R_n(H,X)
  \]
  exist, and, if so, what is its limit as $X\to\infty$?
  For a prescribed growth function $H=H(X)\to\infty$, does
  $R_n(H(X),X)$ converge, and how does its behavior depend on $n$
  and on that growth function?
\end{question}

Embedding degrees in Tables~\ref{tab:crypto-signatures} and~\ref{tab:h4-hits} already
vary substantially in this small grid. Larger searches should track
both $\ord_q(p)$ and the transfer-field size $p^k$ relative to the
largest prime divisor $q$.

The forced rational $2$-torsion rules out anomalous odd good
reductions but also forces divisibility of the group order. Passing to
rational $2$-isogenous quotients over the same finite field does not
change the cofactor statistics: reduction gives an $\F_p$-rational
$2$-isogeny, and isogenous elliptic curves over a finite field have
the same Frobenius trace
\citep[Ex.~5.4(a)]{Silverman-AEC}. Thus new cofactor
behaviour must come from changing the finite-field isogeny class,
for example by varying $n$, $\{a,b,c\}$, the prime set, or the search
bounds.

\paragraph{Reproducibility.}
All numerical tables are produced with Magma V2.29-6~\citep{Magma}.
The Magma sources, row-level data and regeneration instructions are
included with this preprint as ancillary files, with a
\href{https://github.com/chatchawanpan-dev/dn-ecc-crypto-suitability-computation}{GitHub archive}
of the original computation. Separate Python verification scripts
reproduce the triple census and arithmetic diagnostics. They count
the original split cubics and Legendre curves directly at the first
four test primes, and verify the remaining group orders by exact
point-order certificates and the Hasse interval. An independent
Jacobian-arithmetic verifier checks those certificates. The
archived computations assert the Legendre trace formula in the
representative computations, the worked example, and all $1640$ rows
behind Table~\ref{tab:h4-hits}; as an independent check, they evaluate the
character sum directly for $p=1009,10007,100003$.



\begin{thebibliography}{99}

\bibitem{BinduSomanathDas2024}
Bindu V. A., M. Somanath, and R. Das, ``Elliptic Curve Cryptography
using a Diophantine Triple, Constructed through Tetradecagonal
Numbers,'' \emph{International Journal of Intelligent Systems and
Applications in Engineering}, vol. 12, no. 22s, pp. 1524--1529, 2024.
Available: \url{https://ijisae.org/index.php/IJISAE/article/view/6695}.
Accessed: 2026-05-01.

\bibitem{Magma}
Computational Algebra Group, University of Sydney, \emph{Magma
Computational Algebra System}, version V2.29-6. Available:
\url{https://magma.maths.usyd.edu.au/magma/}; release notes:
\url{https://magma.maths.usyd.edu.au/magma/releasenotes/pdf/relv229.pdf}.
Accessed: 2026-05-01.

\bibitem{Dujella2001MTuplesEC}
A. Dujella, ``Diophantine $m$-tuples and elliptic curves,''
\emph{Journal de Th{\'e}orie des Nombres de Bordeaux}, vol. 13,
no. 1, pp. 111--124, 2001.
doi: \href{https://doi.org/10.5802/jtnb.308}{10.5802/jtnb.308}.

\bibitem{Dujella-survey}
A. Dujella, ``There are only finitely many Diophantine quintuples,''
\emph{Journal f\"ur die reine und angewandte Mathematik}, vol. 2004,
no. 566, pp. 183--214, 2004.
doi: \href{https://doi.org/10.1515/crll.2004.003}
{10.1515/crll.2004.003}.

\bibitem{DujellaFuchs}
A. Dujella and C. Fuchs, ``Complete solution of a problem of
Diophantus and Euler,'' \emph{Journal of the London Mathematical
Society}, vol. 71, no. 1, pp. 33--52, 2005.
doi: \href{https://doi.org/10.1112/S002461070400609X}
{10.1112/S002461070400609X}.

\bibitem{Dujella-open}
A. Dujella, ``Diophantine $m$-tuples.'' Available:
\url{https://web.math.pmf.unizg.hr/~duje/dtuples.html}.
Accessed: 2026-05-01.

\bibitem{FrRuck}
G. Frey and H.-G. R\"uck, ``A remark concerning $m$-divisibility and
the discrete logarithm in the divisor class group of curves,''
\emph{Mathematics of Computation}, vol. 62, no. 206, pp. 865--874,
1994. doi: \href{https://doi.org/10.2307/2153546}{10.2307/2153546}.

\bibitem{MOV-attack}
A. Menezes, T. Okamoto, and S. Vanstone, ``Reducing elliptic curve
logarithms to logarithms in a finite field,'' \emph{IEEE Transactions
on Information Theory}, vol. 39, no. 5, pp. 1639--1646, 1993.
doi: \href{https://doi.org/10.1109/18.259647}{10.1109/18.259647}.

\bibitem{Pollard-rho}
J. M. Pollard, ``Monte Carlo methods for index computation (mod
$p$),'' \emph{Mathematics of Computation}, vol. 32, no. 143,
pp. 918--924, 1978.
doi: \href{https://doi.org/10.1090/S0025-5718-1978-0491431-9}
{10.1090/S0025-5718-1978-0491431-9}.

\bibitem{Qiu2025Twists}
D. Qiu, ``On twists of a family of elliptic curves and their
$L$-function,'' \emph{Journal of Comprehensive Pure and Applied
Mathematics}, vol. 3, no. 2, pp. 1--32, 2025.
arXiv: \href{https://arxiv.org/abs/1511.07581}{1511.07581}.

\bibitem{Schoof}
R. Schoof, ``Elliptic curves over finite fields and the computation
of square roots mod $p$,'' \emph{Mathematics of Computation},
vol. 44, no. 170, pp. 483--494, 1985.
doi: \href{https://doi.org/10.1090/S0025-5718-1985-0777280-6}
{10.1090/S0025-5718-1985-0777280-6}.

\bibitem{SemaevAnomalous}
I. A. Semaev, ``Evaluation of discrete logarithms in a group of
$p$-torsion points of an elliptic curve in characteristic $p$,''
\emph{Mathematics of Computation}, vol. 67, no. 221, pp. 353--356,
1998. doi: \href{https://doi.org/10.1090/S0025-5718-98-00887-4}
{10.1090/S0025-5718-98-00887-4}.

\bibitem{Silverman-AEC}
J. H. Silverman, \emph{The Arithmetic of Elliptic Curves}, 2nd ed.,
Graduate Texts in Mathematics, vol. 106. Springer, 2009.
doi: \href{https://doi.org/10.1007/978-0-387-09494-6}
{10.1007/978-0-387-09494-6}.

\bibitem{Smart-anomalous}
N. P. Smart, ``The discrete logarithm problem on elliptic curves of
trace one,'' \emph{Journal of Cryptology}, vol. 12, no. 3,
pp. 193--196, 1999. doi:
\href{https://doi.org/10.1007/s001459900052}{10.1007/s001459900052}.

\bibitem{AuerTop2002Legendre}
R. Auer and J. Top, \lq\lq Legendre elliptic curves over finite fields,''
\emph{Journal of Number Theory}, vol. 95, no. 2, pp. 303--312, 2002.
doi: \href{https://doi.org/10.1006/jnth.2001.2760}
{10.1006/jnth.2001.2760}.

\bibitem{DujellaMikic2014D4}
A. Dujella and M. Miki\'c, \lq\lq On the torsion group of elliptic curves
induced by $D(4)$-triples,''
\emph{Analele \c{S}tiin\c{t}ifice ale Universit\u{a}\c{t}ii
\lq\lq Ovidius'' Constan\c{t}a. Seria Matematic\u{a}},
vol. 22, no. 2, pp. 79--90, 2014.
doi: \href{https://doi.org/10.2478/auom-2014-0033}
{10.2478/auom-2014-0033}.

\bibitem{Zywina2011Koblitz}
D. Zywina, \lq\lq A refinement of Koblitz's conjecture,''
\emph{International Journal of Number Theory},
vol. 7, no. 3, pp. 739--769, 2011.
doi: \href{https://doi.org/10.1142/S1793042111004411}
{10.1142/S1793042111004411}.
arXiv: \href{https://arxiv.org/abs/0909.5280}{0909.5280}.

\end{thebibliography}
\end{document}